\documentclass[12pt,reqno,oneside]{amsart}

\usepackage{amsmath,amsthm,amsfonts,amssymb}
\usepackage{indentfirst}
\usepackage{url}
\usepackage{graphicx}
\usepackage[usenames]{color}

\theoremstyle{plain}
\newtheorem{theorem}{Theorem}[section]

\newtheorem{lem}[theorem]{Lemma}
\newtheorem{prop}[theorem]{Proposition}

\theoremstyle{definition}

\newtheorem{exa}[theorem]{Example}
\newtheorem{obs}[theorem]{Remark}

\numberwithin{equation}{section}

\newcommand{\mc}{C(\lambda, p)}

\newcommand{\MCO}{C_d^o(\lambda,p)}
\newcommand{\MCOd}{C_2^o(\lambda,p)}

\newcommand{\MCI}{C_d^i(\lambda,p)}

\begin{document}

\baselineskip=18pt

\title[Extinction time in growth models subject to geometric catastrophes]{Extinction time in growth models subject to geometric catastrophes}

\author[Valdivino Vargas Junior]{Valdivino Vargas Junior}
\address[Valdivino Vargas Junior]{Institute of Mathematics and Statistics, Federal University of Goias, Campus Samambaia, 
CEP 74001-970, Goi\^ania, GO, Brazil}
\email{vvjunior@ufg.br}
\thanks{F\'abio Machado was supported by CNPq (303699/2018-3) and Fapesp (17/10555-0) and Alejandro Roldan by Universidad de Antioquia.}

\author[F\'abio P. Machado]{F\'abio~Prates~Machado}
\address[F\'abio P. Machado]{Statistics Department, Institute of Mathematics and Statistics, University of S\~ao Paulo, CEP 05508-090, S\~ao Paulo, SP, Brazil.}
\email{fmachado@ime.usp.br}

\author[Alejandro Rold\'an]{Alejandro~Rold\'an-Correa}
\address[Alejandro Rold\'an]{Instituto de Matem\'aticas, Universidad de Antioquia, Calle 67, no 53-108, Medellin, Colombia}
\email{alejandro.roldan@udea.edu.co}

%%%%

\date{\today}

\begin{abstract}
Recently, different dispersion strategies in population models subject to geometric catastrophes have been considered as strategies to improve the chance of po\-pu\-lation's survival. Such dispersion strategies have been contrasted with the strategy where there is no dispersion, comparing the probabilities of survival. In this article, we contrast survival strategies when extinction occurs almost surely, evaluating which strategy prolongs population's life span. Our results allow one to analyze what is the best strategy based on parameters as the probability that each individual exposed to catastrophe survives, the growth rate of the colony, the type of dispersion and the spatial restrictions. 
\end{abstract}

\maketitle

\section{Introduction}
\label{S: Introduction}

When a catastrophe strikes a population, its size is reduced according to some probabi\-li\-ty law. The dispersion of the survivors, considering the spatial restrictions, is a strategy that could help the population to increase its viability. These are the biological and environmental forces that influence the chances of survival. Various stochastic models to represent population growth dynamics subject to catastrophes has been proposed. 

Artalejo~\textit{et al}~\cite{AEL2007} and Brockwell~\cite{B1986, BGR1982}  analyze models for po\-pu\-lations that after catastrophes, the survivors individuals remain together in the same colony. Schinazi~\cite{S2014}, Machado~\textit{et al}~\cite{MRS2015} and Junior~\textit{et al}~\cite{JMR2016,JMR2020, MRV2018} studied models for populations that after catastrophes, individuals disperse trying to make new colonies to improve the odds of their species survival. In all these works, different types of catastrophes and different dispersion strategies were considered. Dispersion holds a central role for both the dynamics and evolution of spatially structured populations. While it can save a small population from local extinction, it also can increase global extinction risk if observed in a very high level. See Ronce~\cite{R2007} for more about dispersal in the biological context.

Recently, Junior \textit{et al}~\cite{JMR2020} analyzed different dispersion strategies in populations subject to geometric catastrophes, to study how these strategies impact the population viability, comparing them with the strategy where there is no dispersion.  Their analysis points to establish which is the best strategy (dispersion or no dispersion), based on the survival probability of the population when some strategy is adopted. For this, at least one of the compared models (model with dispersion and model without dispersion) has to have survival probability greater than zero. However, when both models have the survival probability equal zero the analysis does not give information regarding to the best strategy. In this work, we propose to evaluate which strategy is better when extinction occurs almost surely, considering the mean extinction times. The extinction time of a population is of particular importance  in view of its relevance to the estimation of \lq\lq minimum viable population size\rq\rq\ to guarantee survival for a certain time, see Brockwell~\cite{B1986}.

In Section 2 we present the non dispersion model  proposed in Artalejo~\textit{et al}~\cite{AEL2007} and the models with dispersion proposed in Junior~\textit{et al}~\cite{JMR2020}. Moreover, we added new results for these models. In Section 3 we discuss dispersal strategies for increasing life expectancy. Finally, in Section 4 we prove the results presented in Sections 2.

\section{Models and Results}

\subsection{Geometric catastrophe}

Populations are frequently exposed to catastrophic events that cause massive elimination of their individuals, for example, habitat destruction, environmental disaster, epidemics, etc. A catastrophe can instantly wipe out the entire population or just a part of it. In order to model such events, it is assumed that when a population is hit by a catastrophe, its size is reduced according to some law of probability. For catastrophes that reach the individuals sequentially and the effects of a disaster stop as soon as the first individual survives, if there is any survivor, the appropriate model assume a geometric probability law. That is, if at a catastrophe time the size of the population is $i$, it is reduced to $j$ with probability 
\[ \mu_{ij} = \left\{\begin{array}{ll} (1-p)^i, & j=0\\ p(1-p)^{i-j}, & 1\leq j \leq i,\end{array}\right.\]
where $0<p<1$. The form of $\mu_{ij}$ represents what is called \textit{geometric catastrophe}.

The geometric catastrophe would correspond to cases where the decline in the population is halted as soon as any individual survives the catastrophic event. This
may be appropriate for some forms of catastrophic
epidemics or when the  catastrophe has a sequential propagation effect like in the predator-prey models - the predator kills prey until it becomes satisfied. More examples can be found in Artalejo \textit{et al}~\cite{AEL2007}, Cairns and Pollett~\cite{CairnsPollett}, Economou and Gomez-Corral~\cite{EG2007}, Thierry Huillet~\cite{TH2020} and Kumar \textit{et al}~\cite{KBG2020}.

\subsection{Growth model without dispersion}
Artalejo \textit{et al}~\cite{AEL2007} present a model for a population which sticks together in one colony, without dispersion.  That colony gives birth to new individuals at rate $\lambda>0$,  while geometric catastrophes happen at rate $\mu$.

The population size (number of individuals in the colony) at time $t$ is a continuous time Markov process $\left\{X(t):t\geq 0\right\}$ that we denote by $\mc$. We assume $\mu=1$ and $X(0)=1$. \\
Artalejo \textit{et al}~\cite{AEL2007} use the word \textit{extinction} to describe the event that $X(t) = 0$, for some $t>0$, for a process where state 0 is not an absorbing state. In fact the extinction time here is the first hitting time to the state 0,  
$$\tau_A:=\inf\{t>0:X(t)=0\}.$$ The probability of extinction of $\mc$ is denoted by $\psi_A=\mathbb{P}[\tau_A<\infty].$ Its complement, $1-\psi_A$, is called survival probability. The next result establishes the mean time of extinction for $\mc$.

\begin{theorem}[Artalejo \textit{et al}~\cite{AEL2007}]
\label{th:semdisptime}
For the process $\mc$,     
$$\mathbb{E}[\tau_A]=\left\{
\begin{array}{cl}
\displaystyle\frac{1}{1-p-\lambda p},& \text{if }  p< \displaystyle\frac{1}{\lambda+1};\\ \\
\infty, & \text{if }  p\geq \displaystyle\frac{1}{\lambda+1}.
\end{array}\right.
$$
\end{theorem}

\subsection{Growth models with dispersion and spatial restriction}
	Let $\mathbb{T}_d^+$ be an infinite rooted tree whose vertices
	have degree $d+1$, except the root that has degree $d$. Let us define a process with 
	dispersion on $\mathbb{T}_d^+$, starting from a single colony placed at the root of $\mathbb{T}_d^+$, with just one individual. The number of individuals in a colony grows following a Poisson process of rate $\lambda>0$. To each colony 
	we associate an exponential time of mean 1 that indicates when the geometric catastrophe strikes a colony. The individuals that survived the catastrophe are dispersed between the $d$ neighboring vertices furthest from the root to create new colonies. Among the survivors that go to the same vertex to create a new colony at it, only one succeeds, the others die. So in this case  when a catastrophe occurs in a colony, that colony is replaced by 0,1, ... or $ d $ colonies. We consider two types of dispersion:
	\begin{itemize}
		\item \textbf{Optimal dispersion:} 
		Individuals are distributed in a ordered fashion, from left to right, in order to create the largest possible number of new colonies. If $r$ individuals survive to a catastrophe, then the number of colonies that are created equals $\min\{r,d\}$. Let us denote the process with optimal dispersion by $\MCO$.
		\item 
		\textbf{Independent dispersion:} Each one of the individuals that
		survived the catastrophe picks
		randomly a neighbor vertex and tries to create a new colony at it.
		When the amount of survivors is $r$, the probability of having $y \le \min\{d,r\}$ vertices colonized
		is
		\[ \frac{T(r,y)}{d^r} {{d}\choose{y}}, \]
		where $T(r, y)$  denote the number of surjective functions $f:A \to B$, with $|A| = r$ and
		$|B| = y$.  Let us denote the process with independent dispersion by $\MCI$.
	\end{itemize}

	$\MCO$ and $\MCI$ are continuous-time Markov processes with state space  $\mathbb{N}_0^{\mathbb{T}^d}$. For each particular realization of these processes, we say that it {\it survives} if for any instant of time there is at least one colony somewhere. Otherwise, we say that it {\it dies out }.

	\begin{theorem}[Junior \textit{et al} \cite{JMR2020}]\label{MCOeMCI} Let $\psi_d^o$ and $\psi_d^i$ the extinction probabilities for the processes $\MCO$ and $\MCI$, respectively. Then
	\begin{itemize}
		\item[$(i)$] $\psi_2^o < 1 $ if and only if $p > \frac{1}{\lambda+1}.$
		\item[$(ii)$] $\psi_3^o < 1 $ if and only if
		$p > \frac{\lambda + 1}{2 \lambda^2 + 2 \lambda + 1}$.
		\item[$(iii)$] $\psi_2^i < 1 $ if and only if
		$p > \frac{\lambda + 2}{\lambda^2 + 2 \lambda + 2}.$
		\item[$(iv)$] $\psi_3^i < 1 $ if and only if
		$ p > \frac{\lambda + 3}{2\lambda^2 + 3 \lambda + 3}.$
	\end{itemize}
	\end{theorem}

It is clear that when the survival probability is positive, the mean extinction time for the processes $\MCO$ and $\MCI$ is infinite. In the next results,  we derive the mean extinction time when extinction occurs almost surely, when $d=2$ and $d=3$.  

\begin{theorem}\label{MCOtime}
 Let $\tau_d^o$ the extinction time of the process $\MCO$. 
 \begin{itemize}
 	\item[$(i)$] If $p < \displaystyle\frac{1}{\lambda+1}$, then
 	$\mathbb{E}[\tau_2^o]=
 	\displaystyle\left(1 + \frac{1}{\lambda p}\right)\ln\left(\frac{1-p}{1-p-\lambda p}\right).
 	$\\
 	If $p = \displaystyle\frac{1}{\lambda+1}$, then
 	$\mathbb{E}[\tau_2^o]=\infty.$
 	\item[$(ii)$] If  $p < \displaystyle\frac{\lambda + 1}{2 \lambda^2 + 2 \lambda + 1},$ then  
 	$$\mathbb{E}[\tau_3^o]=
 	\displaystyle
 	\frac{\lambda p +1}{\lambda p} 
 	\sqrt{\frac{p(\lambda+1)}{4+\lambda p-3p}}\,
 	\ln\left[\frac{(2-2p-\lambda p)\sqrt{p(\lambda  +1)}+\lambda p\sqrt{4+\lambda p -3p}}{(2-2p-\lambda p)\sqrt{p(\lambda  +1)}-\lambda p\sqrt{4+\lambda p -3p}}\right].
 	$$
 	If  $p = \displaystyle\frac{\lambda + 1}{2 \lambda^2 + 2 \lambda + 1},$ then  
 	$\mathbb{E}[\tau_3^o]=\infty$.
 \end{itemize}
\end{theorem}

 \begin{theorem}\label{MCItime}
	Let $\tau_d^i$ the extinction time of the process $\MCI$. 
	\begin{itemize}
		\item[$(i)$] If $ p < \displaystyle\frac{\lambda + 2}{ \lambda^2 + 2 \lambda + 2} $, then 
	$$\mathbb{E}[\tau_2^i]=
	\displaystyle\frac{(\lambda+2)(\lambda p+1)}{\lambda p (\lambda + 1 )}\ln\left(\frac{(1 - p)(\lambda +2)}{ \lambda +2 - p (\lambda^2 + 2\lambda + 2)}\right).
	$$
	If $ p = \displaystyle\frac{\lambda + 2}{ \lambda^2 + 2 \lambda + 2} $, then 
	$\mathbb{E}[\tau_2^i]=\infty.$
	\item[$(ii)$] If   $p < \displaystyle\frac{\lambda + 3}{ 2\lambda^2 + 3 \lambda + 3}$, then

	$$\mathbb{E}[\tau_3^i]=
	\displaystyle\frac{( \lambda p + 1)(2 \lambda + 3)(\lambda + 3)}{2h(\lambda, p)}\ln\left[\frac{g(\lambda, p)+h(\lambda, p)}{g(\lambda, p)-h(\lambda, p)}\right], 
	$$
	where 
	\begin{equation}\label{function_g}
	g(\lambda, p)=(\lambda +3)(2\lambda+3-3\lambda p-3p-\lambda^2p),
	\end{equation}
	and
	\begin{equation}\label{function_h}
	h(\lambda, p)=\lambda \sqrt{p(\lambda +1)(p\lambda^2+4\lambda+6-3p)(\lambda +3)}.
	\end{equation}
	If   $p = \displaystyle\frac{\lambda + 3}{ 2\lambda^2 + 3 \lambda + 3}$, then $\mathbb{E}[\tau_3^i]=\infty.$
	
   \end{itemize}
\end{theorem}

\begin{obs}\label{otimo-infinito}
Theorems \ref{MCOtime} and \ref{MCItime} show explicitly the formulas for the mean extinction
times of the processes $\MCO$ and $\MCI$ for $d = 2, 3.$ 
Observe that for $d=\infty$ the only model that makes sense is the optimal model, which corresponds to the scheme with no spatial restriction where each individual that survives a catastrophe, creates its own new
colony independently of everything else. This model is presented in the next section.
\end{obs}

\subsection{Growth model with dispersion but no spatial restrictions.} 
Consider a population of individuals divided into separate colonies. Each colony begins with  
an individual. The number of individuals in each colony increases independently according 
to a Poisson process of rate $\lambda > 0 $.
To each colony we associate an exponential time of mean 1 that indicates when the geometric catastrophe strikes a colony. Each individual that survived the catastrophe 
begins a new colony independently of everything else.
We denote this process by $C_*(\lambda,p)$ and consider it starting from a single colony with just one
individual.\\

For each particular realization of $C_*(\lambda,p)$, we say that it {\it survives} if for any instant of time there is at least one colony somewhere. Otherwise, we say that it {\it dies out.} 
	We denoted by $\psi_*$, the probability of extinction of $C_*(\lambda,p).$ Junior \textit{et al} \cite[Theorem 2.3]{JMR2016} showed that $\psi_*<1$ if and only if $p>\displaystyle\frac{1}{\lambda^2+\lambda+1}$. 
	\\It is clear that when $\psi_*<1$, the mean extinction time of $C_*(\lambda,p)$ is infinite.
	The following theorem establishes the mean time of extinction for $C_*(\lambda,p)$ when $\psi_*=1.$  
	\begin{theorem}\label{th:comdisptime}
		Let $\tau_*$ the extinction time of the process $C_*(\lambda,p)$. Then      
		$$\mathbb{E}[\tau_*]=\left\{
		\begin{array}{ccl}
		\displaystyle1- \frac{p(\lambda+1)^2}{\lambda(\lambda p +1)}\ln\left[1-\frac{\lambda(\lambda p+1)}{(\lambda+1)(1-p)}\right]&,& \text{if }   p<\displaystyle\frac{1}{\lambda^2+\lambda+1};\\ \\
		\infty&, & \text{if }  p\geq\displaystyle\frac{1}{\lambda^2+\lambda+1}.
		\end{array}\right.
		$$
\end{theorem}

Next we compare the dispersal strategies of the models $\MCO$, $\MCI$ and
$C_*(\lambda,p)$ (with spatial dispersion) and $\mc$ (without dispersion) for increasing life expectancy.

\section{Discussion}

\subsection{Non-dispersion vs dispersion with spatial restriction.\\}

By coupling arguments one can see that the extinction times,  $\tau_d^o$ and $\tau_d^i$, are a non-decreasing functions of $d$, $\lambda$ and $p$. Moreover, as the optimal dispersion (due to spatial restriction) maximizes the number of new colonies whenever there are individuals that survived from the latest catastrophe, that type of dispersion is the one which maximizes the extinction time. Thus, $$\mathbb{E}[\tau_d^o]\geq \mathbb{E}[\tau_d^i].$$

Junior \textit{et al} \cite{JMR2020} compute explicitly the extinction probabilities ($\psi_2^o$, $\psi_3^o$, $\psi_2^i$, and $\psi_3^i$) as functions of $\lambda$ and $p$. In particular, they showed that extinction probabilities for the models $\mc$ and $\MCOd$ are equal ($\psi_A=\psi_2^o$).  An interesting question is to determine whether,
when the models $\mc$ and $\MCOd$ die out almost surely, dispersion is an advantage or not to extend the population's life span. This question is answered by the following proposition.

\begin{prop}\label{optimo2}
	Assume $p<\frac{1}{\lambda+1}$.  Then,
	$\mathbb{E}[\tau_A]< \mathbb{E}[\tau_2^o]$ if and only if   
	\begin{equation}\label{eq:optimo2}
	\frac{\lambda p}{(1-p-\lambda p)(1+\lambda p)}<\ln\left(\frac{1-p}{1-p-\lambda p}\right).
	\end{equation}
	Moreover, $\mathbb{E}[\tau_A]= \mathbb{E}[\tau_2^o]$ if and only if we have an equality in (\ref{eq:optimo2}).
\end{prop}
 Proposition~\ref{optimo2} is a consequence of Theorems~\ref{th:semdisptime} and \ref{MCOtime}$(i)$. From Proposition~\ref{optimo2} we can conclude that optimal dispersion is a better strategy compared to
 non-dispersion, when the parameters $(\lambda, p)$ fall in the gray region of Figure~\ref{fig:otimo2}. The opposite (non-dispersion is a better strategy than optimal dispersion) holds in the yellow region. Observe that in the white region $\mathbb{E}[\tau_A] = \mathbb{E}[\tau_2^o] = \infty.$ 
 
\begin{figure}[ht]
	\begin{tabular}{ccc}
		$\lambda$ & \parbox[c]{9cm}{\includegraphics[trim={0cm 0cm 0cm 0cm}, clip, width=8.5cm]{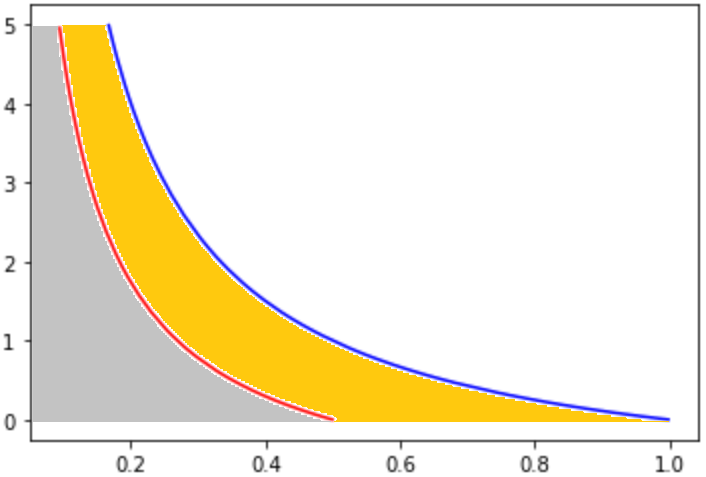}} & %trim={1.1cm 1.3cm 0cm 0cm}, clip, width=9cm
		\begin{tabular}{l}
			\textbf{\textcolor{red}{------}} Equality in (\ref{eq:optimo2})\\ \\ 
			\textbf{\textcolor{blue}{------}}  $p=\frac{1}{\lambda +1}$
		\end{tabular} \\
		& $p$
	\end{tabular}
	\caption{In the gray region, $\mathbb{E}[\tau_A]<\mathbb{E}[\tau_2^o]$. In the yellow region, $\mathbb{E}[\tau_A]>\mathbb{E}[\tau_2^o]$.}
	\label{fig:otimo2}
\end{figure}

\begin{exa}\label{ex_otimo2}
	The processes $C(1,p)$ and $C_2^o(1,p)$ die out if and only if $p\leq1/2$. 
	In this case, solving (\ref{eq:optimo2}) as an equality, we obtain $p_c\approx 0,269059$ and the following statements. 
 \begin{itemize}
 	\item If $0<p<p_c$, then $\mathbb{E}[\tau_A]<\mathbb{E}[\tau_2^o]$.
 	\item If $p=p_c$, then $\mathbb{E}[\tau_A]=\mathbb{E}[\tau_2^o]$.  
 	\item If $p_c<p<1/2$, then  $\mathbb{E}[\tau_2^o]<\mathbb{E}[\tau_A]$. 
    \item If $p\geq1/2$, then  $\mathbb{E}[\tau_2^i]=\mathbb{E}[\tau_A]=\infty$
\end{itemize}  
Observe that the phase transition in $p$ for $\lambda=1$ occurs for all $\lambda>0.$ For $\lambda\geq5$ (not shown in Figure~\ref{fig:otimo2}), we observe that if $p=\frac{1}{\lambda(\lambda+1)}$, then $\mathbb{E}[\tau_A]<\mathbb {E}[\tau_2^o]$, while if $p=\frac{1}{\lambda+2}$, then $\mathbb{E}[\tau_A]>\mathbb{E}[\tau_2^o] $. 
\end{exa}

The next result considers the Artalejo model and the model with independent dispersion with $d=2$ when both models die out almost surely, more precisely when $ p<\min\left\{\frac{1}{\lambda+1},\frac{\lambda + 2}{ \lambda^2 + 2 \lambda + 2}\right\}=\frac{1}{\lambda+1}.$ 

\begin{prop}\label{indepte2} Assume  $p < \frac{1}{\lambda+1}.$ Then $\mathbb{E}[\tau_A]<\mathbb{E}[\tau_2^i]$ if and only if	
	\begin{equation}\label{eq:indepte2}
	\frac{\lambda p (\lambda + 1 )}{(1-p-\lambda p)(\lambda+2)(\lambda p+1)}<
	\ln\left(\frac{(1 - p)(\lambda +2)}{ \lambda +2 - p (\lambda^2 + 2\lambda + 2)}\right).
	\end{equation}
	
Moreover, $\mathbb{E}[\tau_A]= \mathbb{E}[\tau_2^i]$ if and only if we have an equality in (\ref{eq:indepte2}).
		
\end{prop}

Proposition~\ref{indepte2} is a consequence of Theorems~\ref{th:semdisptime} and \ref{MCItime}$(i)$. From Proposition~\ref{indepte2} we can conclude that independent dispersion is a better strategy compared to
non-dispersion, when the parameters $(\lambda, p)$ fall in the gray region of Figure~\ref{fig:indepte2}. The opposite (non-dispersion is a better strategy than independent dispersion) holds in the yellow region. Observe that from Theorems~\ref{th:semdisptime} and \ref{MCItime}$(i)$, we also have that:
	\begin{itemize}
		\item If $\frac{1}{\lambda+1}\leq p<\frac{\lambda + 2}{ \lambda^2 + 2 \lambda + 2}$, then  $\mathbb{E}[\tau_2^i] < \infty$ and $\mathbb{E}[\tau_A] = \infty.$
		\item If $ p\geq\frac{\lambda + 2}{ \lambda^2 + 2 \lambda + 2}$, then $\mathbb{E}[\tau_A] = \mathbb{E}[\tau_2^i] = \infty.$
	\end{itemize}	
Furthermore, Junior \textit{et al}~\cite{JMR2020} showed that the extinction probabilities in the white region of Figure~\ref{fig:indepte2} satisfies $\psi_2^i>\psi_A.$ Thus, in the white region, non-dispersion is a better strategy than independent dispersion.

\begin{figure}[ht]
	\begin{tabular}{ccc}
		$\lambda$ & \parbox[c]{9cm}{\includegraphics[trim={0cm 0cm 0cm 0cm}, clip, width=8.5cm]{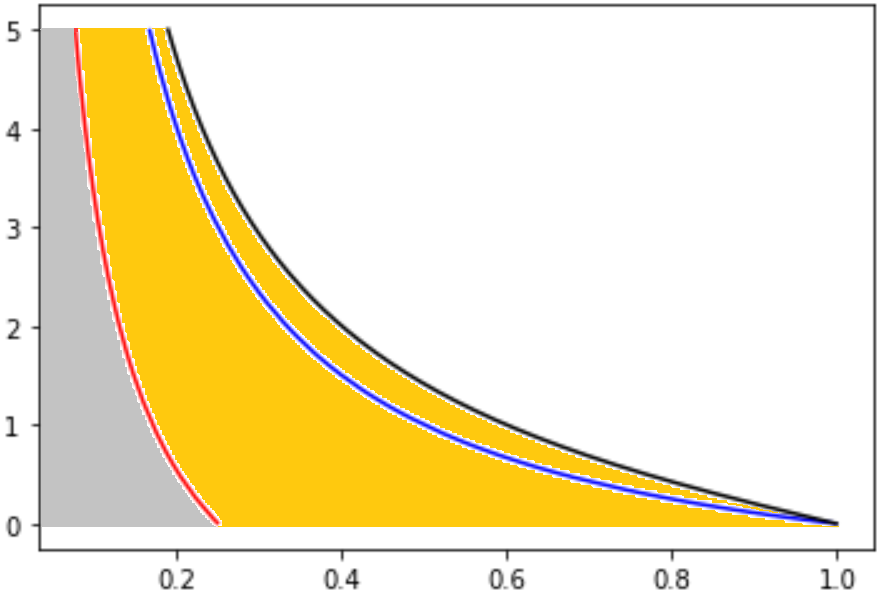}} & 
		\begin{tabular}{l}
			\textbf{\textcolor{red}{------}} Equality in (\ref{eq:indepte2})\\ \\
			\textbf{\textcolor{blue}{------}}  $p=\frac{1}{\lambda +1}$\\ \\
			\textbf{\textcolor{black}{------}} $p=\frac{\lambda + 2}{ \lambda^2 + 2 \lambda + 2}$
		\end{tabular} \\
		& $p$
	\end{tabular}
	\caption{In the gray region, $\mathbb{E}[\tau_A]<\mathbb{E}[\tau_2^i]$. In the yellow region, $\mathbb{E}[\tau_A]>\mathbb{E}[\tau_2^i]$.}
	\label{fig:indepte2}
\end{figure}

\begin{exa}\label{ex_indep2} Both processes, $C(1,p)$ and $C_2^i(1,p)$, die out if and only if $p\leq1/2$. 	
In this case, solving (\ref{eq:indepte2}) as an equality, we obtain $p_c\approx 0.170767$ and the following statements. 
\begin{itemize}
\item If $0<p<p_c$, then $\mathbb{E}[\tau_A]<\mathbb{E}[\tau_2^i]$.
\item If $p=p_c$, then $\mathbb{E}[\tau_A]=\mathbb{E}[\tau_2^i]$.  
\item If $p_c<p<1/2$, then  $\mathbb{E}[\tau_2^i]<\mathbb{E}[\tau_A]<\infty$.
\item If $1/2\leq p< 3/4$, then  $\mathbb{E}[\tau_2^i]<\mathbb{E}[\tau_A]=\infty$.
\item If $p\geq 3/4$, then  $\mathbb{E}[\tau_2^i]=\mathbb{E}[\tau_A]=\infty$.
\end{itemize}  
Observe that the phase transition in $p$ for $\lambda=1$ occurs for all $\lambda>0.$ For $\lambda\geq5$ (not shown in Figure~\ref{fig:indepte2}), we observe that if $p=\frac{1}{\lambda(\lambda+1)}$, then $\mathbb{E}[\tau_A]<\mathbb {E}[\tau_2^i]$, while if $p=\frac{1}{\lambda+2}$, then $\mathbb{E}[\tau_A]>\mathbb{E}[\tau_2^i]$.
\end{exa}	

The next result considers the Artalejo model and the model with optimal dispersion and $d=3$ when both models die out almost surely, more precisely when
$p<\min\left\{\frac{1}{\lambda+1},\frac{\lambda + 1}{2 \lambda^2 + 2 \lambda + 1}\right\}=\frac{\lambda + 1}{2 \lambda^2 + 2 \lambda + 1}.$ 

\begin{prop}\label{optimo3} Assume  $p < \frac{\lambda + 1}{2 \lambda^2 + 2 \lambda + 1}.$  Then,
	$\mathbb{E}[\tau_A]< \mathbb{E}[\tau_3^o]$ if and only if   
	\begin{equation}\label{eq:optimo3}
	\frac{\lambda p\sqrt{4+\lambda p-3p}}{(1-p-\lambda p)(1+\lambda p)\sqrt{p(\lambda+1)}}< 
		\ln\left[\frac{(2-2p-\lambda p)\sqrt{p(\lambda  +1)}+\lambda p\sqrt{4+\lambda p -3p}}{(2-2p-\lambda p)\sqrt{p(\lambda  +1)}-\lambda p\sqrt{4+\lambda p -3p}}\right].
	\end{equation}
Moreover, $\mathbb{E}[\tau_A]= \mathbb{E}[\tau_3^o]$ if and only if we have an equality in (\ref{eq:optimo3}).	
\end{prop}

Proposition~\ref{optimo3} is a consequence of Theorems~\ref{th:semdisptime} and \ref{MCOtime}$(ii)$. From Proposition~\ref{optimo3} we can conclude that optimal dispersion is a better strategy compared to non-dispersion, when the parameters $(\lambda, p)$ fall in the gray region of Figure~\ref{fig:otimo3}. The opposite (non-dispersion is a better strategy than optimal dispersion) holds in the yellow region. Observe that from Theorems~\ref{th:semdisptime} and \ref{MCOtime}$(ii)$, we also have that:
	\begin{itemize}
		\item If $\frac{\lambda + 1}{2 \lambda^2 + 2 \lambda + 1}\leq p<\frac{1}{\lambda+1}$, then  $\mathbb{E}[\tau_A] < \infty$ and $\mathbb{E}[\tau_3^o] = \infty.$  
		\item If $ p\geq\frac{1}{\lambda+1}$, then $\mathbb{E}[\tau_A] = \mathbb{E}[\tau_3^o] = \infty.$
	\end{itemize}	
	Junior \textit{et al}~\cite{JMR2020} showed that the extinction probabilities in the white region of Figure~\ref{fig:otimo3} satisfies $\psi_3^o\leq\psi_A.$ Thus, in the white region, optimal dispersion  is a better strategy than non-dispersion.

\begin{figure}[ht]
	\begin{tabular}{ccc}
		$\lambda$ & \parbox[c]{9cm}{\includegraphics[trim={1.1cm 1.3cm 0cm 0cm}, clip, width=9cm]{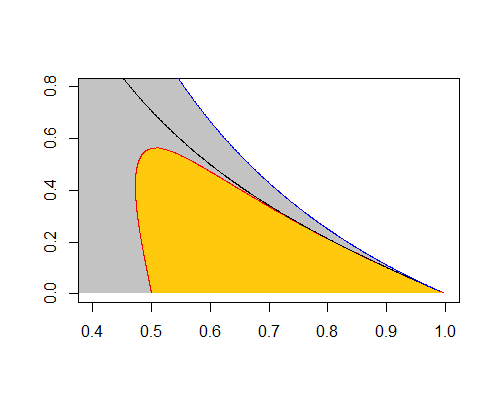}} & 
		\begin{tabular}{l}
			\textbf{\textcolor{red}{------}} Equality in (\ref{eq:optimo3})\\ \\ 
			\textbf{\textcolor{blue}{------}}  $p=\frac{1}{\lambda +1}$\\ \\ 
			\textbf{\textcolor{black}{------}}  $p=\frac{\lambda + 1}{2 \lambda^2 + 2 \lambda + 1}$
		\end{tabular} \\
		& $p$
	\end{tabular}
	\caption{In the gray region, $\mathbb{E}[\tau_A]<\mathbb{E}[\tau_3^o]$. In the yellow region, $\mathbb{E}[\tau_A]>\mathbb{E}[\tau_3^o]$.}
	\label{fig:otimo3}
\end{figure}

\begin{exa}\label{ex_otimo3} Both processes, $C(0.4,p)$ and $C_3^o(0.4,p)$, die out if and only if $p\leq35/53$. 	
	In this case, considering (\ref{eq:optimo3}), we obtain (and define) the critical parameters $p_l$ and $p_u$ such that: 
	\begin{itemize}
		\item If $0<p<p_l$, then $\mathbb{E}[\tau_A]<\mathbb{E}[\tau_3^o]$.
		\item If $p=p_l$, then $\mathbb{E}[\tau_A]=\mathbb{E}[\tau_3^o]$.  
		\item If $p_l<p<p_u$, then  $\mathbb{E}[\tau_3^o]<\mathbb{E}[\tau_A]$. 
		\item If $p=p_u$, then  $\mathbb{E}[\tau_A]=\mathbb{E}[\tau_3^o]$.
		\item If $p_u<p<35/53$, then  $\mathbb{E}[\tau_A]<\mathbb{E}[\tau_3^o]<\infty$.
		\item If $35/53\leq p< 5/7$, then $\mathbb{E}[\tau_A]<\mathbb{E}[\tau_3^o]=\infty$.
		\item If $p\geq5/7$, then  $\mathbb{E}[\tau_A]=\mathbb{E}[\tau_3^o]=\infty$.
		\end{itemize}
Moreover, from numerical approximations we obtain that $p_l\approx 0.4724$ and $p_u\approx 0.6529$. Finally, observe that the phase transition in $p$ for $\lambda = 0.4$ does not occur for all $\lambda>0.$ For example (see Figure~\ref{fig:otimo3}), for $\lambda=0.6$ we have that $\mathbb{E}[\tau_A]<\mathbb {E}[\tau_3^o]$ for all $p<1/(\lambda+1)$.
\end{exa}

The next result considers the Artalejo model and the model with independent dispersion and $d=3$ when both models die out almost surely.

\begin{prop}\label{indepte3} Assume $p<\min\left\{\frac{1}{\lambda+1}, \frac{\lambda + 3}{ 2\lambda^2 + 3 \lambda + 3}\right\}$. Then $\mathbb{E}[\tau_A]<\mathbb{E}[\tau_3^i]$ in and only if  
\begin{equation}\label{eq:indepte3}	
	\frac{2h(\lambda, p)}{(1-p-\lambda p)( \lambda p + 1)(2 \lambda + 3)(\lambda + 3)}<\ln\left[\frac{g(\lambda, p)+h(\lambda, p)}{g(\lambda, p)-h(\lambda, p)}\right], 
\end{equation}
where $g(\lambda, p)$ and $h(\lambda, p)$ are given in (\ref{function_g}) and (\ref{function_h}), respectively. Moreover,  \linebreak $\mathbb{E}[\tau_A]=\mathbb{E}[\tau_3^i]$ if and only if we have an equality in (\ref{eq:indepte3}).
\end{prop}

Proposition~\ref{indepte3} is a consequence of Theorems~\ref{th:semdisptime} and \ref{MCItime}$(ii)$. From Proposition~\ref{indepte3} we can conclude that independent dispersion is a better strategy compared to non-dispersion, when the parameters $(\lambda, p)$ fall in the gray region of Figure~\ref{fig:indepte3}. The opposite (non-dispersion is a better strategy than independent dispersion) holds in the yellow region. Observe that from Theorems~\ref{th:semdisptime} and \ref{MCItime}$(ii)$, we also have that:
	\begin{itemize}
		\item If $\lambda>1$ and $\frac{\lambda + 3}{2 \lambda^2 + 3 \lambda + 3}\leq p<\frac{1}{\lambda+1}$, then  $\mathbb{E}[\tau_A] < \infty$ and $\mathbb{E}[\tau_3^i] = \infty.$  
		\item If $\lambda<1$ and $\frac{1}{\lambda+1} \leq p<\frac{\lambda + 3}{2 \lambda^2 + 3 \lambda + 3}$, then  $\mathbb{E}[\tau_3^i] < \infty$ and $\mathbb{E}[\tau_A] = \infty.$
	\end{itemize}	
	Junior \textit{et al}~\cite{JMR2020} showed that for the extinction probabilities in the white region of Figure~\ref{fig:otimo3} there are two possible behaviors. For the region (I) we have that $\psi_3 ^ i <\psi_A $; in the region (II) we have that $ \psi_3 ^ i> \psi_A. $ 

\begin{figure}[ht]
	\begin{tabular}{ccc}
		$\lambda$ & \parbox[c]{9cm}{\includegraphics[trim={1.1cm 1.3cm 0cm 0cm}, clip, width=9cm]{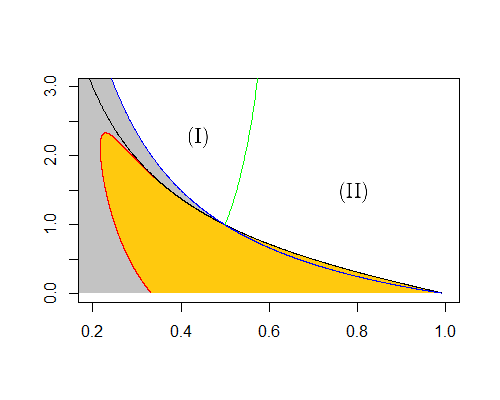}} & 
		\begin{tabular}{l}
			\textbf{\textcolor{red}{------}} Equality in (\ref{eq:indepte3})\\ \\ 
			\textbf{\textcolor{blue}{------}}  $p=\frac{1}{\lambda +1}$\\ \\ 
			\textbf{\textcolor{black}{------}}  $p=\frac{\lambda + 3}{ 2\lambda^2 + 3 \lambda + 3}$
			\\ \\ 
			\textbf{\textcolor{green}{------}}  $p=\frac{2(\lambda + 1)}{ 3\lambda + 5}$
		\end{tabular} \\
		& $p$
	\end{tabular}
	\caption{In the gray region, $\mathbb{E}[\tau_A]<\mathbb{E}[\tau_3^i]$. In the yellow region, $\mathbb{E}[\tau_A]>\mathbb{E}[\tau_3^i]$.}
	\label{fig:indepte3}
\end{figure}

	\begin{exa}\label{ex_indepte3} Both processes, $C(2.2,p)$ and $C_3^i(2.2,p)$, die out if and only if $p\leq65/241$. 	
		In this case, considering (\ref{eq:indepte3}), we obtain $p_l\approx 0.2174$ and $p_u\approx 0.2594$ such that: 
		\begin{itemize}
			\item If $0<p<p_l$, then $\mathbb{E}[\tau_A]<\mathbb{E}[\tau_3^i]$.
			\item If $p=p_l$, then $\mathbb{E}[\tau_A]=\mathbb{E}[\tau_3^i]$.  
			\item If $p_l<p<p_u$, then  $\mathbb{E}[\tau_3^i]<\mathbb{E}[\tau_A]$. 
			\item If $p=p_u$, then  $\mathbb{E}[\tau_A]=\mathbb{E}[\tau_3^i]$.
			\item If $p_u<p<65/241$, then  $\mathbb{E}[\tau_A]<\mathbb{E}[\tau_3^i]<\infty$.
			\item If $65/241\leq p< 5/16$, then $\mathbb{E}[\tau_A]<\mathbb{E}[\tau_3^i]=\infty$.
			\item If $p\geq5/16$, then  $\mathbb{E}[\tau_A]=\mathbb{E}[\tau_3^i]=\infty$.
		\end{itemize}  
The phase transition in $p$ observed for $\lambda =2.2$ does not occur for all $\lambda>0.$ For example (see Figure~\ref{fig:indepte3}), for $\lambda=3$ we have that $\mathbb{E}[\tau_A]<\mathbb {E}[\tau_3^i]$ for all $p<1/(\lambda+1)$.
\end{exa}

\subsection{Non-dispersion vs dispersion without spatial restriction.\\}

The following result establishes a comparison between the mean extinction times for $\mc$, the model without dispersion and $C_*(\lambda, p)$, the model with dispersion without spatial restriction. We restrict our attention to where both models die out almost surely, more precisely when  $p<\min\left\{\frac{1}{\lambda+1},\frac{1}{\lambda^2 +  \lambda + 1}\right\}=\frac{1}{\lambda^2 +  \lambda + 1}.$
	\begin{prop}\label{disp}
		Assume $p<\frac{1}{\lambda^2+\lambda+1}$.  Then,
		$\mathbb{E}[\tau_A]< \mathbb{E}[\tau_*]$ if and only if   
		\begin{equation}\label{eq:disp}
			\frac{-\lambda (\lambda p+1)}{(\lambda+1)(1-p-\lambda p)}>\ln\left[1-\frac{\lambda(\lambda p+1)}{(\lambda+1)(1-p)}\right].
		\end{equation}
		Moreover, $\mathbb{E}[\tau_A]= \mathbb{E}[\tau_*]$ if and only if we have an equality in (\ref{eq:disp}).
	\end{prop}
	Proposition~\ref{disp} is a consequence of Theorems~\ref{th:semdisptime} and \ref{th:comdisptime}. From Proposition~\ref{disp} we can conclude that  dispersion is a better strategy compared to
	non-dispersion, when the parameters $(\lambda, p)$ fall in the gray region of Figure~\ref{fig:semrestri}. The opposite (non-dispersion is a better strategy than dispersion) holds in the yellow region. Observe that from Theorems~\ref{th:semdisptime} and \ref{th:comdisptime}, we also have that:
	\begin{itemize}
		\item If $\frac{1}{\lambda^2+\lambda+1}\leq p<\frac{1}{\lambda + 1}$, then  $\mathbb{E}[\tau_A] < \mathbb{E}[\tau_*] = \infty.$ 
		\item If $ p\geq\frac{1}{\lambda + 1}$, then $\mathbb{E}[\tau_A] = \mathbb{E}[\tau_*] = \infty.$
	\end{itemize}	
	Junior \textit{et al}~\cite[Remark 2.7]{JMR2016} showed that the extinction probabilities in the white region of Figure~\ref{fig:semrestri} satisfies $\psi_*<\psi_A.$ Thus, in the white region, dispersion is a better strategy than non-dispersion. 
	\begin{figure}[ht]
		\begin{tabular}{ccc}
			$\lambda$ & \parbox[c]{9cm}{\includegraphics[trim={1.1cm 1.3cm 0cm 0cm}, clip, width=9cm]{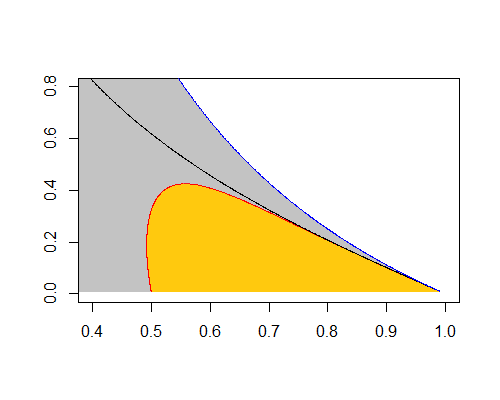}} & 
			\begin{tabular}{l}
				\textbf{\textcolor{red}{------}} Equality in (\ref{eq:disp})\\ \\ 
				\textbf{\textcolor{blue}{------}}  $p=\frac{1}{\lambda +1}$\\ \\ 
				\textbf{\textcolor{black}{------}}  $p=\frac{1}{\lambda^2 + \lambda + 1}$
			\end{tabular} \\
			& $p$
		\end{tabular}
		\caption{In the gray region, $\mathbb{E}[\tau_A]<\mathbb{E}[\tau_*]$. In the yellow region, $\mathbb{E}[\tau_A]>\mathbb{E}[\tau_*]$.}
		\label{fig:semrestri}
	\end{figure}
\begin{exa}\label{ex_disp} Both processes, $C(0.4,p)$ and $C_*(0.4,p)$, die out if and only if $p\leq25/39$. 	
	In this case, considering (\ref{eq:disp}), we obtain $p_l\approx 0.5209$ and $p_u\approx 0.6118$, therefore 
	\begin{itemize}
		\item If $0<p<p_l$, then $\mathbb{E}[\tau_A]<\mathbb{E}[\tau_*]$.
		\item If $p=p_l$, then $\mathbb{E}[\tau_A]=\mathbb{E}[\tau_*]$.  
		\item If $p_l<p<p_u$, then  $\mathbb{E}[\tau_*]<\mathbb{E}[\tau_A]$. 
		\item If $p=p_u$, then  $\mathbb{E}[\tau_A]=\mathbb{E}[\tau_*]$.
		\item If $p_u<p<25/39$, then  $\mathbb{E}[\tau_A]<\mathbb{E}[\tau_*]<\infty$.
		\item If $25/39\leq p< 5/7$, then $\mathbb{E}[\tau_A]<\mathbb{E}[\tau_*]=\infty$.
		\item If $p\geq5/7$, then  $\mathbb{E}[\tau_A]=\mathbb{E}[\tau_*]=\infty$.
	\end{itemize}  
The phase transition in $p$ observed for $\lambda = 0.4$ does not occur for all $\lambda>0.$ For example (see Figure~\ref{fig:semrestri}), for $\lambda=0.5$ we have that $\mathbb{E}[\tau_A]<\mathbb {E}[\tau_*]$ for all $p<1/(\lambda+1)$.
\end{exa}

\begin{obs}
	From Remark~\ref{otimo-infinito} and the monotonicity in $d$ (by coupling arguments)  of $\tau_d^o$, we have that $\mathbb {E}[\tau_*]\geq\mathbb{E}[\tau_3^o]\geq\mathbb{E}[\tau_2^o].$ Thus, the yellow region of Figure~\ref{fig:semrestri} is contained in the yellow region of Figure~\ref{fig:otimo3}, which in its turn is contained in the yellow region of Figure~\ref{fig:otimo2}. Consequently, the regions of the parametric space $p \times \lambda$ where $\mathbb{E}[\tau_A]>\mathbb{E}[\tau_d^o]$ (the yellow regions) tends from above to the yellow region of Figure~\ref{fig:semrestri} as $d$ tends to infinity.
\end{obs}

\subsection{Conclusion.\\}
In general, observe that the model without dispersion (with only one colony) has a catastrophe rate of 1 while the models with dispersion (multiple colonies) has a catastrophe rate of $n$ whenever there are $n$ colonies. Moreover, a catastrophe is more likely to wipe out a smaller colony than a larger one. On the other hand
multiple colonies give multiple chances for survival and this may be a critical advantage
of the multiple colonies model over the single colony model. Also note that in the models with dispersion and spatial restriction, during the dispersion some individuals could end up at the same spatial location. In this case, all but one individuals will die. As a result there is a trade-off: On the one hand, dispersion creates independent populations and thus promotes survival. On the other hand, dispersion could lead to death due to competition for space.

Therefore, our results show that dispersion may be or may not be an advantage for prolongs population's life span depending (not trivially) on the dispersion type, the spatial restrictions,  the growth rate of the colonies, and the probability that each individual exposed to catastrophe survives.

\section{Proofs}

\begin{lem}\label{lemaux}
Let $(Y_t)_{t\geq 0}$ a continuous time branching process, where each particle survives an exponential time of rate 1 and right before death produces a random number of particles  with probability generating function
\[f(s)=\sum_{k=0}^{\infty}p_ks^k.\]
Suppose that $Y_0=1$ and $f'(1)\leq 1$.  Let $\tau=\inf\{t>0:Y_t=0\}$,  the extinction time of the process $(Y_t)_{t\geq 0}$. 
\begin{itemize}
	\item[$(i)$] If {$p_2\neq 0$ and $p_k= 0$ for $k\geq 3$,} then
		$$\mathbb{E}[\tau]=
		\left\{\begin{array}{cl}
		\displaystyle\frac{1}{p_2}\ln\left(\frac{p_0}{p_0-p_2}\right) & \text{, if } f'(1)<1,\\ \\
		\infty & \text{, if } f'(1)=1.
		\end{array}
		\right.$$
	\item[$(ii)$] If {$p_3\neq 0$ and $p_k= 0$ for $k\geq 4$,} then
		$$\mathbb{E}[\tau]=
		\left\{\begin{array}{cl}
		\displaystyle\frac{1}{\sqrt{4p_0p_3+(p_2+p_3)^2}}\ln\left[\frac{2p_0-p_2-p_3+\sqrt{4p_0p_3+(p_2+p_3)^2}}{2p_0-p_2-p_3-\sqrt{4p_0p_3+(p_2+p_3)^2}}\right] & \text{, if } f'(1)<1,\\ \\
		\infty & \text{, if } f'(1)=1.
		\end{array}
		\right.$$
	{\item[$(iii)$] If $p_0=\beta$ and $p_n=\alpha c^n$ for $n\geq 1$, where $\alpha, \beta$ and $c$ are positive constants, then
	$$\mathbb{E}[\tau]=
	\left\{\begin{array}{cl}
	1-\displaystyle\frac{1-\beta}{c}\ln\left[1-\frac{c}{\beta}\right] & \text{, if } f'(1)<1,\\ \\
	\infty & \text{, if } f'(1)=1.
	\end{array}
	\right.$$}

\end{itemize}
\end{lem}
\begin{proof}[Proof of Lemma \ref{lemaux}]
If $f'(1)\leq1$, then $\mathbb{P}[\tau<\infty]=1$. Thus, from Narayan~\cite{PN1982}, we have that	
	\begin{equation}\label{Narayan}
	\mathbb{E}[\tau]=\displaystyle\int_0^1\frac{1-y}{f(y)-y}dy.
	\end{equation}

\begin{itemize}
	\item[(i)]
	If {$p_2\neq 0$ and $p_k= 0$ for $k\geq 3$,} from (\ref{Narayan}) we obtain that 
	$$\begin{array}{lll}
	\mathbb{E}[\tau]%&=&\displaystyle\int_0^1\frac{1-y}{p_0+p_1y+p_2y^2-y}dy\\ \\
	&=&\displaystyle\int_0^1\frac{1-y}{p_0+p_1y+(1-p_0-p_1)y^2-y}dy\\ \\
	%&=&\displaystyle\int_0^1\frac{1-y}{p_0(1+y)(1-y)+p_1y(1-y)- y(1-y)}dy\\ \\
	&=&\displaystyle\int_0^1\frac{1}{p_0-(1-p_0-p_1)y}dy\\ \\
	&=&\displaystyle\int_0^1\frac{1}{p_0-p_2y}dy\\ \\
	&=&\left\{\begin{array}{cl}
	\displaystyle\frac{1}{p_2}\ln\left(\frac{p_0}{p_0-p_2}\right) & \text{, if } f'(1)<1,\\ \\
	\infty & \text{, if } f'(1)=1.
	\end{array}
	\right.

	\end{array}$$
	\item[$(ii)$] 
	If {$p_3\neq 0$ and $p_k= 0$ for $k\geq 4$,} from (\ref{Narayan}) we obtain that 
	$$\begin{array}{lll}
	\mathbb{E}[\tau]
	&=&\displaystyle\int_0^1\frac{1-y}{p_0+p_1y+p_2y^2+(1-p_0-p_1-p_2)y^3-y}dy\\ \\
	%&=&\displaystyle\int_0^1\frac{1-y}{p_0(1-y)(1+y+y^2)+p_1y(1-y)(1+y)+p_2y^2(1-y)-y(1-y)(1+y)}dy\\ \\
	%&=&\displaystyle\int_0^1\frac{1}{p_0(1+y+y^2)+p_1y(1+y)+p_2y^2-y(1+y)}dy\\ \\
	&=&\displaystyle\int_0^1\frac{1}{p_0-(1-p_0-p_1)y-(1-p_0-p_1-p_2)y^2}dy\\ \\
	&=&\displaystyle\int_0^1\frac{1}{p_0-(p_2+p_3)y-p_3y^2}dy \\ \\
	&=&\left\{\begin{array}{cl}
	\displaystyle\frac{1}{\sqrt{4p_0p_3+(p_2+p_3)^2}}\ln\left[\frac{2p_0-p_2-p_3+\sqrt{4p_0p_3+(p_2+p_3)^2}}{2p_0-p_2-p_3-\sqrt{4p_0p_3+(p_2+p_3)^2}}\right] & \text{, if } f'(1)<1,\\ \\
	\infty & \text{, if } f'(1)=1,
	\end{array}
	\right.
	\end{array}$$
where the last equality has been obtained using
	$$\int \frac{dy}{ay^2+by+c}= \displaystyle\frac{1}{\sqrt{b^2-4ac} }\ln\left[\frac{2ay+b-\sqrt{b^2-4ac}}{2ay+b+\sqrt{b^2-4ac}}\right]+ \text{constant},   
$$
 when $b^2-4ac>0$, see  Prudnikov~\textit{et al}~\cite[Eq. 1.2.8.13]{PBY2002}. In our case, $a=-p_3$, $b=-(p_2+p_3)$ and $c=p_0.$

{\item[$(iii)$] If $p_0=\beta$ and $p_n=\alpha c^n$ for $n\geq 1$, we have
\begin{equation}
\beta=\frac{1-c-\alpha c}{1-c}
\end{equation}
and	 
$$f(s)=\sum_{k=0}^{\infty}p_ks^k=\beta+\frac{\alpha c s}{1-cs}.$$
Thus, 
\begin{eqnarray*}
	\mathbb{E}[\tau]
	&=&\int_0^1 \frac{(1-y)(1-cy)}{cy^2-(\beta +c)y+\beta}\, dy\\ \\
	&=&\int_0^1 \left(1+\frac{1-\beta}{\beta -cy}\right)dy\\ \\
	&=&\left\{\begin{array}{cl}
		1-\displaystyle\frac{1-\beta}{c}\ln\left(1-\frac{c}{\beta}\right) & \text{, if } f'(1)<1,\\ \\
		\infty & \text{, if } f'(1)=1.
	\end{array}
	\right.
\end{eqnarray*}
}
\end{itemize}
\end{proof}

In order to prove Theorems~\ref{MCOtime}, \ref{MCItime} and \ref{th:comdisptime}, observe that  the probability distribution of the number of survivors right after the catastrophe (but before the dispersion) is given by
	\[\mathbb{P}(N=0) =\beta, \,\mathbb{P}(N = n) = \alpha c^n, n =1,2,\ldots,\] where
	\begin{equation}\label{beta_lambda_c}
\beta = \frac{1-p}{\lambda p + 1}, \ \alpha = \frac{(\lambda + 1)p}{\lambda ( \lambda p +1)} \hbox{ and } \ c = \frac{\lambda}{\lambda + 1}.
	\end{equation} 
For details see Machado~\textit{et al}~\cite[Section 2.2]{MRV2018}.

\begin{proof}[Proof of Theorem~\ref{MCOtime}] 
Let $Z_t$ be the number of colonies at time $t$ in the model $\MCO$. Observe that $Z_t$ is a continuous-time branching process with  $Z_0=1$. Each particle (colony) in $Z_t$ survives an exponential time of rate 1 and right before death produces $k\leq d$ particles (colonies are created right after a catastrophe) with probability $p_k$ given by

$$p_k=\left\{
\begin{array}{cl}
\beta & \text{, if }  k=0;\\ 
\alpha c^k & \text{, if } 1\leq k<d; \\ 
1-\beta-\displaystyle\frac{\alpha c(1-c^{d-1})}{1-c} & \text{, if } k= d. 
\end{array}\right.
$$	
Moreover, $\tau_d^o=\inf\{t>0:Z_t=0\}$.	 \\	
	
$\bullet$ For $d=2$, we have that  
	\[
	p_0 = \mathbb{P}(N=0) = \beta, \ p_1 = \mathbb{P}(N=1)= \alpha c \hbox{ and } p_2 = 1- \beta - \alpha c.
	\]
Furthermore, the condition $p< \frac{1}{\lambda +1}$  is equivalent to $p_1+2p_2<1$.  Thus, from Lemma~\ref{lemaux}$(i)$, we have that  	
	$$\begin{array}{lll}
	\mathbb{E}[\tau_2^o]&=&\displaystyle\frac{1}{p_2}\ln\left(\frac{p_0}{p_0-p_2}\right)\\ \\
	&=&\left(1 + \frac{1}{\lambda p}\right)\ln\left(\frac{1-p}{1-p-\lambda p}\right),
	\end{array}$$
where the last line has been obtained using (\ref{beta_lambda_c}).\\

For $p= \frac{1}{\lambda +1}$, we have that  $p_1+2p_2=1$.  Thus, from Lemma~\ref{lemaux}$(i)$, it follows that  $	\mathbb{E}[\tau_2^o]=\infty$.\\

$\bullet$ For $d=3$, we have that 
	\[
	p_0 = \mathbb{P}(N=0) = \beta, \ p_1 = \mathbb{P}(N=1)= \alpha c, \ p_2 = \mathbb{P}(N=2)= \alpha c^2, \hbox{ and } p_3 = 1- \beta - \alpha c-\alpha c^2.
	\]
Furthermore, the condition $p< \frac{\lambda +1}{2\lambda^2+2\lambda+1}$  is equivalent to $p_1+2p_2+3p_3<1$. Thus, from Lemma~\ref{lemaux}$(ii)$, we have that  		
	  	
	$$\begin{array}{lll}
	\mathbb{E}[\tau_3^o]&=&\displaystyle\frac{1}{\sqrt{4p_0p_3+(p_2+p_3)^2}}\ln\left[\frac{2p_0-p_2-p_3+\sqrt{4p_0p_3+(p_2+p_3)^2}}{2p_0-p_2-p_3-\sqrt{4p_0p_3+(p_2+p_3)^2}}\right]\\\\
	&=&
	\displaystyle
	\frac{\lambda p +1}{\lambda p} 
	\sqrt{\frac{p(\lambda+1)}{4+\lambda p-3p}}\,
	\ln\left[\frac{(2-2p-\lambda p)\sqrt{p(\lambda  +1)}+\lambda p\sqrt{4+\lambda p -3p}}{(2-2p-\lambda p)\sqrt{p(\lambda  +1)}-\lambda p\sqrt{4+\lambda p -3p}}\right],
	\end{array}$$
where the last line has been obtained using (\ref{beta_lambda_c}).\\

If $p= \frac{\lambda +1}{2\lambda^2+2\lambda+1}$, we have that $p_1+2p_2+3p_3=1$.  Thus, from Lemma~\ref{lemaux}$(ii)$, it follows that  $	\mathbb{E}[\tau_3^o]=\infty$.
\end{proof}

\begin{proof}[Proof of Theorem~\ref{MCItime}] Analogously to the proof of Theorem~\ref{MCOtime}. In this case, 
	$$p_k=\left\{
	\begin{array}{cl}
	\beta & \text{, if }  k=0;\\ 
	\alpha\dbinom{d}{k}\displaystyle\sum_{n=k}^{\infty} T(n,k)\left(\frac
	{c}{d}\right)^n   & \text{, if } 1\leq k<d; \\ 
	1- \displaystyle\sum_{j=0}^{d-1}p_j & \text{, if } k= d. 
	\end{array}\right.
	$$

$\bullet$ If $d=2$, we have that 	
	\[
	p_0 = \beta, \ p_1 = \frac{2\alpha c}{2-c} \hbox{ and } p_2 = 1- \beta - \frac{2\alpha c}{2-c}.
	\]

$\bullet$ If $d=3$, we have that 
	\[
	p_0 = \beta,  p_1 = \frac{3\alpha c}{3-c}, p_2= \frac{6\alpha c^2}{(3-2c)(3-c)}  \hbox{ and } p_3 = 1- \beta- \frac{3\alpha c}{3-c}-\frac{6\alpha c^2}{(3-2c)(3-c)}.
	\]

\end{proof}

\begin{proof}[Proof of Theorem~\ref{th:comdisptime}] Analogously to the proof of Theorem~\ref{MCOtime}. In this case, 
$p_0 =\beta$  and  $p_k = \alpha c^k$ for $k \geq 1.$
\end{proof}

\section{Acknowledgments} The authors are thankful for the anonymous referee for a careful reading and many useful suggestions that helped to improve the paper.

\end{document}